\documentclass[11pt,oneside]{amsart}
\usepackage[T1]{fontenc}
\usepackage[latin9]{inputenc}
\usepackage{geometry}
\usepackage{color}
\usepackage[english]{babel}
\usepackage{verbatim}
\usepackage{amsthm}
\usepackage{amstext}
\usepackage{esint}
\usepackage[unicode=true,pdfusetitle,
 bookmarks=true,bookmarksnumbered=false,bookmarksopen=false,
 breaklinks=false,pdfborder={0 0 1},backref=page,colorlinks=true]
 {hyperref}

\makeatletter
\numberwithin{equation}{section}
\numberwithin{figure}{section}
\theoremstyle{plain}
\newtheorem{thm}{\protect\theoremname}
  \theoremstyle{plain}
  \newtheorem{fact}[thm]{\protect\factname}
  \theoremstyle{plain}
  \newtheorem{lem}[thm]{\protect\lemmaname}
  \theoremstyle{plain}
  \newtheorem{prop}[thm]{\protect\propositionname}
  \theoremstyle{remark}
  \newtheorem{rem}[thm]{\protect\remarkname}

\newcommand{\ra}{\rightarrow}

\newcommand{\cG}{{\mathcal G}}

\newcommand{\cO}{{\mathcal O}}

\newcommand{\bN}{{\mathbb N}}

\newcommand{\bR}{{\mathbb R}}
\newcommand{\bC}{{\mathbb C}}
\newcommand{\bZ}{{\mathbb Z}}

\newcommand{\bH}{{\mathbb H}}

\newcommand{\norm}[1]{\left\|#1\right\|}
\newcommand{\set}[1]{\left\{#1\right\}}
\newcommand{\pa}[1]{\left(#1\right)}
\newcommand{\av}[1]{\left|#1\right|}

\makeatother

  \providecommand{\factname}{Fact}
  \providecommand{\lemmaname}{Lemma}
  \providecommand{\propositionname}{Proposition}
  \providecommand{\remarkname}{Remark}
\providecommand{\theoremname}{Theorem}

\begin{document}

\title{Duke's Theorem for subcollections}

\author{Menny Aka}

\author{Manfred Einsiedler}
\begin{abstract}
We combine effective mixing and Duke's Theorem on closed geodesics
on the modular surface to show that certain subcollections of the
collection of geodesics with a given discriminant still equidistribute.
These subcollections are only assumed to have sufficiently large total
length without any further restrictions.
\end{abstract}

\address{Section de mathématiques\\
EPFL\\
Station 8 - Bât. MA\\
CH-1015 Lausanne\\
Switzerland }

\email{menashe-hai.akkaginosar@epfl.ch}

\address{Departement Mathematik\\
ETH Zürich\\
Rämistrasse 101\\
8092 Zürich\\
Switzerland}

\email{manfred.einsiedler@math.ethz.ch}

\maketitle

\section{Introduction}

Duke's Theorem, in our context, is concerned with the equidistribution
of closed geodesics on the modular surface $Y_{0}(1):={\rm SL}_{2}(\bZ)\setminus\bH$
(and its unit tangent bundle). To give the necessary background and
its statement, we follow the introduction of \cite{ELMV}. The reader
is referred to there for the definitions of the classical notions
that we use below.

A non-zero integer $d$ is called a  \textit{discriminant} if there
exist $a,b,c\in\bZ$ such that $d=b^{2}-4ac$. For any non-square
positive discriminant $d$ one can associate (see \cite[\S 1.2]{ELMV})
a collection $\cG_{d}$ of $h(d)$ closed geodesics on $X:=T^{1}(Y_{0}(1))\cong{\rm SL}_{2}(\bZ)\setminus{\rm SL}_{2}(\bR)$,
the unit tangent of the modular surface, where $h(d)$ is the class
number of the order $\cO_{d}:=\bZ[\frac{d+\sqrt{d}}{2}]$ (see \cite[\S 2.1]{ELMV}).
Duke's Theorem asserts that the set $\cG_{d}$ becomes equidistributed
as $d\ra+\infty$ amongst positive non-square discriminants. The aim
of this paper is to deduce a similar theorem for subcollections of
$\cG_{d}$ of sufficiently large total length without any further
restrictions. In order to give a precise formulation and relate this
work to previous results we first record the following facts:
\begin{fact}
\label{thm:Fact on duke ingridents}Let $d$ be a positive non-square
discriminant and $\cO_{d}:=\bZ[\frac{d+\sqrt{d}}{2}]$ be the order
of discriminant $d$. We have:\end{fact}
\begin{enumerate}
\item $\av{\cG_{d}}=\av{{\rm Pic}(\cO_{d})}$ where ${\rm Pic}(\cO_{d})$
is the ideal class group of $\cO_{d}$.
\item The length of any $\phi\in\cG_{d}$ is equal to ${\rm Reg}(\cO_{d})$,
the regulator of $\cO_{d}$.%

\item The total length of the collection $\cG_{d}$ is ${\rm Reg}(\cO_{d})\cdot\av{\cG_{d}}=d^{\frac{1}{2}+o(1)}$\label{enum: total length}.\end{enumerate}
\begin{proof}
See \cite[ \S2]{ELMV}.
\end{proof}
Let $G={\rm SL_{2}}(\bR)$, $\Gamma$ be a finite-index congruence
subgroup of ${\rm SL_{2}}(\bZ)$ and let $G$ act on $X=\Gamma\setminus G$
by $g.\Gamma x=\Gamma xg^{-1}$. Let $C_{c}^{\infty}(X)$ denote the
space of infinitely differentiable functions with compact support
on $X$ and $\mu_{X}$ denote the unique $G$-invariant probability
measure on $X$$ $. Throughout this paper, given a measure $\nu$
on a measurable space $X$ and a $\nu$-measurable function $f$,
we set $\nu(f):=\int_{X}fd\nu$. For a closed geodesic $\phi$, we
denote its length by $l(\phi)$ and let $l(I_{d}):=\sum_{\phi\in I_{d}}l(\phi)$.
By Fact \ref{thm:Fact on duke ingridents}, we have $l(I_{d})=\av{I_{d}}Reg(\cO_{d})$.
We let $\mu_{\phi}$ denote the normalized arc-length measure along
$\phi$, that is, for $f\in C_{c}^{\infty}(X)$ we set 
\[
\mu_{\phi}(f)=\frac{1}{l(\phi)}\int_{0}^{l(\phi)}f(a_{t}x)dt,\text{ for some }x\in\phi,
\]
where $a_{t}$ denote the geodesic flow (see $\S$\ref{sec:Preliminaries}).
For any $I_{d}\subseteq\cG_{d}$ let $\mu_{I_{d}}:=\frac{1}{\av{I_{d}}}\sum_{\phi\in I_{d}}\mu_{\phi}$
denote the normalized measure supported on $I_{d}$ and $\mu_{d}:=\mu_{\cG_{d}}$.
Finally, given a sequence of subcollections $\mathcal{I}=\set{I_{d_{k}}}$
we let $\varphi_{\mathcal{I}}(k)=\frac{l(\cG_{d_{k}})}{l(I_{d_{k}})}$.
In this note, we prove the following:
\begin{thm}
\label{thm:subcollection Duke}Let $\mathcal{I}=\set{I_{d_{k}}}$
be a sequence of subcollections such that $\psi(k):=\frac{\varphi_{\mathcal{I}}(k)}{\log(d_{k})}$
tends to $0$ as $k\ra\infty$. Then, for any $f\in C_{c}^{\infty}(X)$
we have 
\[
\av{\mu_{I_{d_{k}}}(f)-\mu_{X}(f)}\leq C(f)\psi(k)^{\frac{1}{2}}
\]
where $C$ is a constant depending only on $f$. In particular, \textup{$\mu_{I_{d_{k}}}$
equidistribute to $\mu_{X}$}.
\end{thm}

Note that the only assumption on $I_{d}$ is about its total length.
Theorem \ref{thm:subcollection Duke} answers a question raised in
\cite[Remark 6.1.]{MV2006ICM}. In order to discuss a stronger variant
of Theorem \ref{thm:subcollection Duke} and to put Theorem \ref{thm:subcollection Duke}
in the context of remark \cite[Remark 6.1.]{MV2006ICM}, one should
contrast Theorem \ref{thm:subcollection Duke} with the results in
\cite{Popa2006,MH2006}. To explain these results, note that after
choosing a base point, $\cG_{d}$ inherits a structure from ${\rm Pic}(\cO_{d})$
(i.e. $\cG_{d}$ is a ${\rm Pic}(\cO_{d})$-torsor, see \cite[\S 2]{ELMV}
). In \cite{Popa2006,MH2006,HarcosThesis}, the authors establish
the equidistribution of subcollections that correspond to subgroups
of $H_{d}<{\rm Pic}(\cO_{d})$ with $[{\rm Pic}(\cO_{d}):H_{d}]\gg d^{a}$
for some $a<\frac{1}{2827}$. In other words, they establish equidistribution
of much smaller subcollections which are restricted by some \textquotedbl{}algebraic\textquotedbl{}
condition. We note that these results, do not imply Theorem \ref{thm:subcollection Duke}.
First, in the context of Heegner points (which is the framework of
the result in \cite{MH2006}), Theorem \ref{thm:subcollection Duke}
is clearly false for arbitrary subcollections. Indeed, restricting
to points which lie in a certain part of positive measure of the total
space, yields subcollections $I_{d}$ with $\av{I_{d}}\geq C\av{\cG_{d}}$,
for some $0<C<1$ which do not equidistribute. Moreover, in the context
of closed geodesics, arbitrary subcollections with $\frac{l(\cG_{d})}{l(I_{d})}\ll d{}^{a}$
for some $a>0$ do not necessarily equidistribute: following a construction
that was outlined to us by Elon Lindenstrauss, for any $a>0$ we construct
in Section \ref{sec:Elon's-construction} subcollections with $\frac{l(\cG_{d})}{l(I_{d})}\ll d^{a}$
which do not equidistribute, and in fact give positive mass to an
arbitrary fixed periodic orbit. This construction uses subcollections
of $\cG_{d}$ for which ${\rm Reg}(\cO_{d})=c\log(d)$. While writing
this note we found that in an upcoming preprint \cite{BK2014}, Bourgain
and Kontorovich, construct subcollections with $\frac{l(\cG_{d})}{l(I_{d})}\ll d{}^{a}$
that stay uniformly bounded. Moreover, using sieve methods, they manage
to construct uniformly bounded subcollections along a sequence that
involves only fundamental discriminants.

It is an interesting question to decide whether Theorem \ref{thm:subcollection Duke}
holds for smaller subcollections under the assumption ${\rm Reg}(\cO_{d})\gg d^{\epsilon}$
for some $\epsilon>0$. (For a stronger conjecture and a related discussion,
see also \cite[Conjecture 1.9]{ELMVDuke}.)

It is important to note that a stronger, but non-effective, equidistribution
result on subcollections follows from \cite{ELMV}. Indeed, the fact
that $l(\cG_{d})=d^{\frac{1}{2}+o(1)}$ is the only information that
is used in \cite{ELMV} to deduce that the limiting measure has maximal
entropy and hence is equal to $\mu_{X}$. Therefore, the same argument
implies that any subcollections with $l(I_{d})=d^{\frac{1}{2}+o(1)}$
also equidistribute (see mainly \cite[Proposition 3.6]{ELMV}).

Apart from the effectivity in the Theorem \ref{thm:subcollection Duke},
we also remark that the following variant of Theorem \ref{thm:subcollection Duke}
cannot be deduced from \cite{ELMV}, i.e. using the above entropy
argument. Our method uses as input an effective Duke's Theorem, i.e.,
the effective equidistribution of $\cG_{d}$ with $d^{-\gamma}$ savings
for some $\gamma>0$ (see $\S2.2$). Note that by \cite{Popa2006,MH2006,HarcosThesis},
a similar effective Theorem exists for collections $\mathcal{H}_{d}\subset\cG_{d}$
supported on cosets of subgroups $H_{d}<{\rm Pic}(\cO_{d})$ with
$[{\rm Pic}(\cO_{d}):H_{d}]\gg d^{a}$ for some $a<\frac{1}{2827}$,
with $d^{-\gamma(a)}$ savings for some $\gamma(a)>0$ (see \cite[Corollary 1.4]{HarcosThesis}).
Therefore, with the exact same proof, it follows that Theorem \ref{thm:subcollection Duke}
holds for any subcollection $\mathcal{I}=\set{I_{d_{k}}\subset\mathcal{H}_{d_{k}}}$
with $\tilde{\psi}(k):=\frac{\tilde{\varphi}_{\mathcal{I}}(k)}{\log(d_{k})}$
where $\tilde{\varphi}_{\mathcal{I}}(k)=\frac{l(\mathcal{H}_{d_{k}})}{l(I_{d_{k}})}$
instead of $\psi(k)$. 

This note is organized as follows: Theorem \ref{thm:subcollection Duke}
is obtained by a simple application of effective mixing in conjunction
with effective version of Duke's Theorem. In hindsight, a similar
argument is used in \cite{VenkateshSED}. We review these ingredients
in $\S2$ and give the proof of Theorem \ref{thm:subcollection Duke}
in $\S3$. Section $4$ is devoted for the construction of large but
non-equidistributing subcollections as discussed above.

\subsection*{Acknowledgments}

M.A. would like to thank Paul Nelson for fruitful conversations and
patient explanation of the analytic methods behind effective statements
of Duke's Theorem and to Ilya Khayutin for many conversations. We
also want to thank Elon Lindenstrauss for outlining the construction
that appears in $\S4.$ 

M.A. acknowledges the support of ISEF and Advanced Research Grant
228304 from the ERC. While working on this project the authors visited
the IIAS and its hospitality is great acknowledged.

\section{Preliminaries\label{sec:Preliminaries}}

As above, let $G={\rm SL_{2}}(\bR),\, a_{t}:=\pa{\begin{array}{cc}
e^{\frac{t}{2}} & 0\\
0 & e^{-\frac{t}{2}}
\end{array}}$, $\Gamma$ be a finite-index congruence subgroup of ${\rm SL_{2}}(\bZ)$.
Then, $G$ acts on $X=\Gamma\setminus G$ by $g.\Gamma x=\Gamma xg^{-1}$.
Note a left invariant metric $d_{G}$ on $G$ induces a metric on
$X$ which we denote by $d_{X}$ (\cite[\S 9.3.2]{EW10}). It is well
known (see e.g. \cite[\S 9.4.2]{EW10}) that under the identification
of $\Gamma\setminus G\cong T^{1}(Y_{0}(1))$ the action of $a_{t}$
corresponds to the geodesic flow on $X$. We denote by $C_{c}^{\infty}(X)\oplus\bC\subset C^{\infty}(X)$
the space of compactly supported smooth functions modulo the constants
and by $C_{0}^{\infty}(X):=\set{f\in C_{c}^{\infty}(X)\oplus\bC:\int_{X}fd\mu_{X}=0}$.
We define $f_{T}$ by 
\[
f_{T}(x):=\frac{1}{T}\int\limits _{0}^{T}f(a_{t}.x)dt.
\]
Finally, let $0\leq\theta<\frac{1}{2}$ and assume that the unitary
representation of $G$ on $L_{0}^{2}(\Gamma\setminus G)$ does not
weakly contain any complementary series with parameter $\geq\theta$
(for $\Gamma={\rm SL_{2}}(\bZ)$ this holds with $\theta=0$, (the
tempered case)).

\subsection{Effective mixing}
\begin{lem}
\label{lem:Mixing betta} For any $f\in C_{0}^{\infty}(X),\, t\in\bR,\,\epsilon>0$,
we have 
\[
\av{\langle f,a_{t}\circ f\rangle}\ll_{\epsilon}S_{1}(f)^{2}e^{t\pa{\pa{\theta-\frac{1}{2}}+\epsilon}}
\]
where $S_{1}$ is a Sobolev norm.\end{lem}
\begin{proof}
This assertion is proved in \cite[Theorem \S 9.1.2]{VenkateshSED}
with an explicit Sobolev norm or in \cite[page 216]{HoweTanBook}.
In fact, for our argument any bound of the form $\ll S_{1}(f)^{2}e^{t\beta}$
for some $\beta<0$ will suffice.\end{proof}
\begin{prop}
\label{prop:estimating f_T^2}For any real valued $f\in C_{0}^{\infty}(X)$
we have 
\[
\mu_{X}\pa{\av{f_{T}}^{2}}\ll\frac{S_{1}(f)^{2}}{T}
\]
where $S_{1}$ is the same as in Lemma \ref{lem:Mixing betta}.\end{prop}
\begin{proof}
We have
\[
\mu_{X}\pa{\av{f_{T}}^{2}}=\int\limits _{X}\av{f_{T}}^{2}d\mu_{X}=\langle f_{T},f_{T}\rangle=\frac{1}{T^{2}}\int\limits _{0\leq t,s\leq T}\langle f,f\circ a_{t-s}\rangle dsdt={\rm (*)}.
\]
Let $B=\set{(t,s):0\leq t\leq T,0\leq s\leq t}$ and note that since
$f$ is real-valued and $a_{t-s}\circ$ is a unitary operator, we
have $\langle f,f\circ a_{t-s}\rangle=\langle a_{s-t}\circ f,f\rangle=\langle f,f\circ a_{s-t}\rangle$
which implies that
\[
{\rm (*)}=\frac{2}{T^{2}}\int\limits _{B}\langle f,f\circ a_{t-s}\rangle dsdt\leq\frac{2}{T^{2}}\int\limits _{B}\av{\langle f,f\circ a_{t-s}\rangle}dsdt\ll\frac{S_{1}(f)^{2}}{T^{2}}\int\limits _{0}^{T}\int\limits _{0}^{t}e^{\beta\pa{t-s}}dsdt
\]
where $\beta:=\frac{\theta-\frac{1}{2}}{2}<0$. A direct computation
yields:

\[
\int\limits _{0}^{T}\int\limits _{0}^{t}e^{\beta\pa{t-s}}dsdt=\frac{1}{\beta}\int\limits _{0}^{T}\pa{e^{\beta t}-1}dt=\frac{1}{\beta^{2}}\pa{e^{\beta T}-1-\beta T}
\]
and recalling that $\beta<0$ we have reached the claim above:
\[
{\rm (*)}\ll\frac{S_{1}(f)^{2}}{\av{\beta}T}.
\]

\end{proof}

\subsection{Effective Duke's Theorem}

The following theorem is now known as Duke's Theorem \cite{Duke88}:
\begin{thm}
\label{thm:effective duke}There exist $ $a $\gamma>0$ and a sobolev
norm $S_{2}$, such that for any $f\in C_{c}^{\infty}(X)$ we have
\[
\av{\mu_{d}(f)-\mu_{X}(f)}\leq S_{2}(f)d^{-\gamma}
\]
where $S_{2}(f)$ is the Sobolev norm on $C_{c}^{\infty}(X)$ (whose
further properties are discussed below). 
\end{thm}

As we want to apply Theorem \ref{thm:effective duke} to $\av{f_{T}}^{2}$,
we need to bound the growth rate of $S_{2}(\av{f_{T}}^{2})$. To this
end, note that 
\begin{enumerate}
\item There exists another sobolev norm $S_{3}$ such that for any $f_{1},f_{2}\in C_{0}^{\infty}(X)\oplus\bC$
we have $S_{2}(f_{1}f_{2})\ll S_{3}(f_{1})S_{3}(f_{2})$.\label{enu:multiplic sobolev}
\item For any $f\in C_{0}^{\infty}(X)\oplus\bC$ and $g\in G$, $S_{3}(g.f)\ll\norm{g}^{\kappa}S_{3}(f)$
for some $\kappa>0$, where $ $$\norm{g}$ denotes the operator norm
$Ad(g^{-1}):{\rm Lie(G)\ra Lie(G)}$, $X\mapsto g^{-1}Xg$.\label{enu: g action} 
\end{enumerate}
For the proof of Theorem \ref{thm:effective duke} and the properties
of the above Sobolev norms we refer the reader to \cite[Theorem 4.6]{EMLV11}
and the references therein, in particular to \cite[\S 2.9 and \S 6]{VenkateshSED}. 

We thus have:
\begin{lem}
There exists an $\alpha>0$ such that for any $T>0$, and any $f\in C_{c}^{\infty}(X)\oplus\bC$
we have\label{Lem:square estimate} 
\[
S_{2}(\av{f_{T}}^{2})\ll S_{3}(f)^{2}e^{\alpha T}.
\]
\end{lem}
\begin{proof}
This readily follow from properties (\ref{enu:multiplic sobolev})
and (\ref{enu: g action}). Indeed, first use that that $S_{2}(\av{f_{T}}^{2})\ll S_{3}(f_{T})S_{3}(f_{T})$.
Further, by the convexity of the norm $S_{3}$ and Jensen's inequality,
we have 
\[
S_{3}(f_{T})\leq\frac{1}{T}\int_{0}^{T}S_{3}(f\circ a_{t})dt\ll\frac{S_{3}(f)}{T}\int_{0}^{T}e^{\kappa t}dt\ll S_{3}(f)e^{\kappa T},
\]
and the lemma follows.
\end{proof}

\section{Proof of Theorem \ref{thm:subcollection Duke}}

For simplicity, we write $I_{k}=I_{d_{k}}$ and $\mu_{k}=\mu_{I_{k}}$.
Fix $f\in C_{c}^{\infty}(X)\oplus\bC$ and set $c=\mu_{X}(f)$ and
by abuse of notation, let $c$ also denote the constant function $c\cdot1_{X}$.
As we aim to estimate $\av{\mu_{k}(f)-c}$, note first that 
\[
\mu_{k}(f)-c=\mu_{k}(f-c)=\mu_{k}\pa{\pa{f-c}_{T}}
\]
where the first equality follows since $\mu_{k}$ is a probability
measure and the second since $\mu_{k}$ is supported on closed geodesics. 

By the Cauchy-Schwarz inequality, we have 
\begin{eqnarray}
\pa{\frac{l(I_{k})}{l(\cG_{d_{k}})}\mu_{k}\pa{\pa{f-c}_{T}}}^{2}\leq\pa{\frac{l(I_{k})}{l(\cG_{d_{k}})}}^{2} & {\mu_{k}\pa{1_{X}^{2}}}\cdot{\mu_{k}\pa{\av{\pa{f-c}_{T}}^{2}}}\label{eq:first}\\
\leq\pa{\frac{l(I_{k})}{l(\cG_{d_{k}})}}^{2}\mu_{k}\pa{\av{\pa{f-c}_{T}}^{2}} & \leq\frac{l(I_{k})}{l(\cG_{d_{k}})}\mu_{d_{k}}\pa{\av{\pa{f-c}_{T}}^{2}} & =\pa{*}
\end{eqnarray}
where the last inequality follows since the positivity of $\av{\pa{f-c}_{T}}^{2}$
implies that 
\[
\frac{l(I_{k})}{l(\cG_{d_{k}})}\mu_{k}\pa{\av{\pa{f-c}_{T}}^{2}}\leq\mu_{d_{k}}\pa{\av{\pa{f-c}_{T}}^{2}}.
\]
Now we apply Theorem \ref{thm:effective duke}. Note that $\av{\pa{f-c}_{T}}^{2}$
does not have compact support, but it is eventually constant since
$\av{\pa{f-c}_{T}}^{2}-c^{2}$ has compact support. Noting that $\mu_{d}$
and $\mu_{X}$ are probability measures, we can apply Theorem \ref{thm:effective duke}
to estimate $\mu_{d_{k}}\pa{\av{\pa{f-c}_{T}}^{2}}$ and get that
\begin{equation}
\pa{*}\leq\frac{l(I_{k})}{l(\cG_{d_{k}})}\pa{\mu_{X}\pa{\av{\pa{f-c}_{T}}^{2}}+d_{k}^{-\gamma}S_{2}\pa{\av{\pa{f-c}_{T}}^{2}-c^{2}}}=\pa{**}.\label{eq:second}
\end{equation}
Note that $S_{2}\pa{\av{\pa{f-c}_{T}}^{2}-c^{2}}\ll S_{2}\pa{\av{\pa{f-c}_{T}}^{2}}+\norm{f}_{\infty}^{2}$.
Now, as $f-c$ has mean zero we can apply $ $Proposition \ref{prop:estimating f_T^2}
to estimate $\mu_{X}\pa{\av{\pa{f-c}_{T}}^{2}}$ and Lemma \ref{Lem:square estimate}
to bound $S_{2}\pa{\av{\pa{f-c}_{T}}^{2}}$, in order to get 

\[
\pa{**}\ll_{f}\frac{l(I_{k})}{l(\cG_{d_{k}})}\pa{S_{1}(f-c)^{2}T^{-1}+S_{3}(f-c)^{2}d_{k}^{-\gamma}e^{\alpha T}+d_{k}^{-\gamma}\norm{f}_{\infty}^{2}}.
\]
Putting all of the above together and choosing $T=\eta\log(d_{k})$,
we have 
\begin{equation}
\frac{l(I_{k})}{l(\cG_{d_{k}})}\pa{\mu_{k}\pa{f}-\mu_{X}\pa{f}}^{2}\ll_{f}S_{1}(f-c)^{2}\eta^{-1}\log(d_{k})^{-1}+S_{3}(f-c)^{2}d_{k}^{\eta\alpha-\gamma}+d_{k}^{-\gamma}\norm{f}_{\infty}^{2}.\label{eq:last estimate}
\end{equation}
Choosing $\eta<\frac{\gamma}{\alpha}$ and multiplying both sides
by $\varphi_{\mathcal{I}}(k)=\frac{l(\cG_{d_{k}})}{l(I_{k})}$, we
get with $\psi(k)=\frac{\varphi_{\mathcal{I}}(k)}{\log(d_{k})}$ that
\[
\pa{\mu_{k}\pa{f}-\mu_{X}\pa{f}}^{2}\ll_{f}\psi(k)
\]
as claimed.

\section{Large but non-equidistributing subcollections\label{sec:Elon's-construction}}

Recall that $\cO_{d}:=\bZ[\frac{d+\sqrt{d}}{2}]$, the unique order
of discriminant $d$. The following construction was outlined to us
by E. Lindenstrauss:
\begin{thm}
\label{thm:Elon's example}Let $\set{d_{k}}_{k\in K}\nearrow\infty$
be any sequence with ${\rm Reg}(\cO_{d_{k}})\ll\log(d_{k})$. Given
$a>0$ and a fixed periodic orbit $P$, there exist subcollections
$I_{d_{k}}\subset\cG_{d_{k}}$ with $l(I_{d_{k}})\gg d_{k}^{\frac{1}{2}-a}$
such that any weak-{*} limit of a subsequence of $\mu_{I_{d_{k}}}$
gives a positive mass to $P$ and in particular, the sequence $\set{\mu_{I_{d_{k}}}}_{k\in\bN}$
does not equidistribute.\end{thm}
\begin{rem}
Such sequences of discriminants do exist and even exist in any given
fixed real quadratic field (see e.g. \cite[\S 6 ]{McM2007}). 
\end{rem}
Let $P$ be a periodic orbit and note that since $P$ is compact,
it has a uniform injectivity radius which we denote by ${\rm inj}(P)$.
For any $r<{\rm inj}(P)$ we let $U_{r}=\set{x\in X:d_{X}(x,P)<r}$. 
\begin{lem}
\label{lem:Epsilon keeps around} For any small enough $r>0$ and
$y\in U_{r}$ there exists an interval $I$ of length $\asymp-\log(r)$
such that for any $t\in I$, we have $a_{t}.y\in U_{r^{\frac{1}{2}}}$.\end{lem}
\begin{proof}
Let $x\in P$ such that $d_{X}(x,y)<r$ for some $r$ that will be
determined momentarily. Denote $x=\Gamma g_{1},y=\Gamma g_{2}$ and
$h=\left(\begin{array}{cc}
a & b\\
c & d
\end{array}\right)\in G$ with $\Gamma g_{1}h=\Gamma g_{2}$. Fix a norm on $M_{2\times2}(\bR)$,
say $\norm{\left(\begin{array}{cc}
a & b\\
c & d
\end{array}\right)}={\rm max}\pa{\av{a},\av{b},\av{c},\av{d}}$ and restrict it to $G$. We know that the resulting metric $d_{\norm{}}$
on $G$ is bi-Lipschitz equivalent to $d_{G}$ (say $\kappa d_{\norm{}}\leq d_{G}\leq Kd_{\norm{}}$),
and for any $z=\Gamma g_{0}\in P$ and $r<\min\pa{{\rm inj}(P),1}$
the projection $B_{r}^{G}(e)\ra B_{r}^{X}(x),\,\, g\mapsto\Gamma g_{0}g$
is an isometry between $d_{G}$ and $d_{X}$. Thus by assumption $\norm{h}\leq\frac{r}{\kappa}$
and since $a_{t}.y=\Gamma g_{1}a_{t}^{-1}a_{t}ha_{t}^{-1}=a_{t}.x\pa{a_{t}ha_{t}^{-1}}$
we have $d_{X}(a_{t}.x,a_{t}.y)\leq K\norm{a_{t}ha_{t}^{-1}}$. As
$a_{t}.x\in P$ we have to show that there is an interval $I$ of
length $\asymp-\log(r)$ such that $t\in I$ implies 
\begin{equation}
\norm{a_{t}ha_{t}^{-1}}\leq\frac{r^{\frac{1}{2}}}{K}.\label{eq: norm of the conj}
\end{equation}
Since $a_{t}ha_{t}^{-1}=\left(\begin{array}{cc}
a & be^{t}\\
e^{-t}c & d
\end{array}\right)$, and by assumption $\av{a-1},\av{d-1},\av{b},\av{c}\leq\frac{r}{\kappa}\ll\frac{r^{\frac{1}{2}}}{K}$
(where the last inequality holds for small enough $r$), (\ref{eq: norm of the conj})
amounts to $e^{t}\av{b}\leq e^{t}\frac{r}{\kappa}\leq\frac{r^{\frac{1}{2}}}{K}$
and $e^{-t}\av{c}\leq e^{-t}\frac{r}{\kappa}\leq\frac{r^{\frac{1}{2}}}{K}$.
Thus, for small enough $r$, we have $d_{X}(a_{t}.x,a_{t}.y)\leq r^{\frac{1}{2}}$
if and only if
\[
e^{t}\leq\frac{\kappa r^{-\frac{1}{2}}}{K}\text{ and }e^{-t}\leq\frac{\kappa r^{-\frac{1}{2}}}{K}
\]
if and only if $t\leq\log(\frac{\kappa}{K})-\frac{1}{2}\log(r)$ and
$-t\leq\log(\frac{\kappa}{K})-\frac{1}{2}\log(r)$. As for small enough
$r$, $I=[-\log(\frac{\kappa}{K})+\frac{1}{2}\log(r),\log(\frac{\kappa}{K})-\frac{1}{2}\log(r)]$
has length $\asymp-\log(r)$ we are done.
\end{proof}

\begin{lem}
\label{lem:bump functions}Let $P$ be a fixed closed geodesic and
$0\leq r\leq{\rm min(1,}{\rm inj}(P))$. There exists a function $f_{r}=f(r,P)$
such that
\begin{enumerate}
\item $\forall x\in X$, $0\leq f_{r}(x)\leq1$, 
\item ${\rm supp}(f_{r})\subset U_{r}$, $f_{r}|_{P}\equiv\text{1 }$. 
\item $S_{2}(f_{r})\ll r^{-b}$ for some $b>0$ (where $S_{2}$ is as in
Theorem \ref{thm:effective duke}),
\item $\mu_{X}(f_{r}):=\int f_{r}d\mu_{X}\asymp r^{c}$ for some $c>0$.
\end{enumerate}
\end{lem}
\begin{proof}
As $0\leq r\leq{\rm inj}(P)$, this construction takes place in the
compact region of $X$. Therefore, to estimate the Sobolev norm any
standard Sobolev norm on $\bR^{3}$ will do. The most obvious construction
works. Namely, note that the set $\tilde{U}_{r}=P\cdot\set{u^{+}(s)u^{-}(s):\av{s}<r}$
is a subset of $U_{r}$, and we can use the standard bump function
$\Theta(x)=\exp(\frac{-1}{\pa{x-r}^{2}}$) on $(-r,r)$ to define
$f_{r}:\tilde{U}_{r}\ra\bR$ by $f_{r}(\Gamma g_{P}a_{t}u^{+}(s_{1})u^{+}(s_{2}))=\Theta(s_{1})\Theta(s_{2})$
where $ $$P=\set{\Gamma g_{P}a_{t}}_{0\leq t\leq l(P)}$ and $\av{s_{2}},\av{s_{2}}<r$.
One easily checks that $f_{r}$ has the desired properties.
\end{proof}

\begin{proof}[Proof of Theorem \ref{thm:Elon's example}]
 For simplicity we restrain from mentioning injectivity radius issues
any further as these may always be resolved by taking some variables
to be large/small enough.

Let $\gamma$ be as in Theorem \ref{thm:effective duke} and fix $a>0$
and a periodic orbit $P$. Let $\eta=\eta(a)>0$ that will be determined
later. Applying Theorem \ref{thm:effective duke} to the functions
$f_{k}:=f(d_{k}^{-\eta},P)$, which are provided by Lemma \ref{lem:bump functions},
we get that there exist $C_{1},C_{2}>0$ such that 
\[
\mu_{d_{k}}(f_{k})\geq\mu_{X}(f_{k})-d_{k}^{-\gamma}S_{2}(f_{k})\geq C_{1}d_{k}^{-c\eta}-C_{2}d_{k}^{-(\gamma-b\eta)}.
\]
Setting $\eta$ so small such that $\gamma-b\eta\geq c\eta$, i.e.
$\eta\geq\frac{\gamma}{b+c}$, we have
\[
\mu_{d_{k}}(f_{k})\gg d_{k}^{-c\eta}.
\]
For a closed geodesic $\phi$ and $f\in C_{c}^{\infty}(X)$ we set
$\phi(f)$ to be the line integral of $f$ along $\phi$. Recall that
$\mu_{d_{k}}(f)=l(\cG_{d_{k}})^{-1}\sum_{\phi\in\cG_{d_{k}}}\phi(f)$,
and note that since $0\leq f_{k}\leq1$, for any $\phi\in\cG_{d_{k}}$
and any $\epsilon>0$, we have $\phi(f)\leq{\rm Reg}(\cO_{d_{k}})\ll d_{k}^{\epsilon}$
since by assumption ${\rm Reg}(\cO_{d_{k}})\ll\log(d_{k})$. Therefore,
if we let $I_{d_{k}}$ denote the subcollection of all the elements
of $\cG_{d_{k}}$ that intersect the support of $f_{k}$, for any
$\epsilon>0$ we have 
\[
d_{k}^{-c\eta}\ll\mu_{d_{k}}(f_{d_{k}})\ll l(\cG_{d_{k}})^{-1}\av{I_{d_{k}}}d_{k}^{\epsilon}=d_{k}^{-\pa{\frac{1}{2}+o(1)-\epsilon}}\av{I_{d_{k}}}.
\]
Thus, by choosing $\epsilon(a)$ and $\eta(a)$ accordingly, we have
$\av{I_{d_{k}}}\gg d_{k}^{\pa{\frac{1}{2}+o(1)-\epsilon}-c\eta}\gg d_{k}^{\frac{1}{2}-a}$
.

Let $\nu$ be a weak-{*} limit of $\mu_{I_{d_{k}}}$ and we claim
that $\nu(P)>0$. It is enough to show that there exists a $C>0$
such that for any large enough $k$, and any small enough $r$ we
have $\mu_{I_{d_{k}}}(U_{r})>C$. Let $U^{k}=U_{d_{k}^{-\frac{\eta}{2}}}$
and as $d_{k}\nearrow\infty$ it is clear that $\cap_{k}U^{k}=P$.
Thus it is enough to verify that for any $k_{0}$, $\mu_{I_{d_{k}}}(U^{k_{0}})>C$
for $k\gg0$. Fix $k_{0}\in\bN$; for any $k\geq k_{0}$ any element
of $\phi\in I_{d_{k}}$ intersects ${\rm supp}(f_{k})$ by definition,
and therefore there is a point $x\in\phi$ with $d(x,P)\leq d_{k}^{-\eta}$.
By Lemma \ref{lem:Epsilon keeps around} there exists an interval
$I_{k}$ of length $\eta\log(d_{k})$ such that for any $t\in I_{k}$
we have $a_{t}.x\in\phi\cap U^{k}\subset\phi\cap U^{k_{0}}$. Since
the length of any element of $I_{d_{k}}$ is ${\rm Reg}(\cO_{d_{k}})\leq c_{1}\log(d_{k})$
for some constant $c_{1}$, any element of $I_{d_{k}}$ spends at
least $\frac{\eta}{c_{1}}$ of its length in $U_{k_{0}}$. It follows
that $\nu(U_{k_{0}})>\frac{\eta}{c_{1}}>0$ and so that $\nu(P)>0$
and the claim follows.
\end{proof}

\author{\bibliographystyle{plain}
\bibliography{partDuke}

\begin{thebibliography}{10}

\bibitem{BK2014}
Jean Bourgain and Alex Kontorovich.
\newblock Beyond expansion ii: Low-lying fundamental geodesics.
\newblock {\em In preparation}.

\bibitem{Duke88}
W.~Duke.
\newblock Hyperbolic distribution problems and half-integral weight {M}aass
  forms.
\newblock {\em Invent. Math.}, 92(1):73--90, 1988.

\bibitem{ELMVDuke}
Manfred Einsiedler, Elon Lindenstrauss, Philippe Michel, and Akshay Venkatesh.
\newblock Distribution of periodic torus orbits on homogeneous spaces.
\newblock {\em Duke Math. J.}, 148(1):119--174, 2009.

\bibitem{EMLV11}
Manfred Einsiedler, Elon Lindenstrauss, Philippe Michel, and Akshay Venkatesh.
\newblock Distribution of periodic torus orbits and {D}uke's theorem for cubic
  fields.
\newblock {\em Ann. of Math. (2)}, 173(2):815--885, 2011.

\bibitem{ELMV}
Manfred Einsiedler, Elon Lindenstrauss, Philippe Michel, and Akshay Venkatesh.
\newblock The distribution of closed geodesics on the modular surface, and
  {D}uke's theorem.
\newblock {\em Enseign. Math. (2)}, 58(3-4):249--313, 2012.

\bibitem{EW10}
Manfred Einsiedler and Thomas Ward.
\newblock Ergodic theory - with a view towards number theory.
\newblock 2010.

\bibitem{HarcosThesis}
Gergely Harcos.
\newblock Subconvex bounds for automorphic l-functions and applications.
\newblock This is an unpublished dissertation available at
  http://www.renyi.hu/~gharcos/ertekezes.pdf.

\bibitem{MH2006}
Gergely Harcos and Philippe Michel.
\newblock The subconvexity problem for {R}ankin-{S}elberg {$L$}-functions and
  equidistribution of {H}eegner points. {II}.
\newblock {\em Invent. Math.}, 163(3):581--655, 2006.

\bibitem{HoweTanBook}
Roger Howe and Eng-Chye Tan.
\newblock {\em Nonabelian harmonic analysis}.
\newblock Universitext. Springer-Verlag, New York, 1992.
\newblock Applications of ${{\rm{S}}L}(2,{{\bf{R}}})$.

\bibitem{McM2007}
Curtis~T. McMullen.
\newblock Uniformly {D}iophantine numbers in a fixed real quadratic field.
\newblock {\em Compos. Math.}, 145(4):827--844, 2009.

\bibitem{MV2006ICM}
Philippe Michel and Akshay Venkatesh.
\newblock Equidistribution, {$L$}-functions and ergodic theory: on some
  problems of {Y}u.\ {L}innik.
\newblock In {\em International {C}ongress of {M}athematicians. {V}ol. {II}},
  pages 421--457. Eur. Math. Soc., Z\"urich, 2006.

\bibitem{Popa2006}
Alexandru~A. Popa.
\newblock Central values of {R}ankin {$L$}-series over real quadratic fields.
\newblock {\em Compos. Math.}, 142(4):811--866, 2006.

\bibitem{VenkateshSED}
Akshay Venkatesh.
\newblock Sparse equidistribution problems, period bounds and subconvexity.
\newblock {\em Ann. of Math. (2)}, 172(2):989--1094, 2010.

\end{thebibliography}
}
\end{document}